\UseRawInputEncoding
\documentclass[12pt, reqno]{amsart}
\usepackage{amssymb}
\usepackage{amsfonts}
\usepackage{amssymb,amsmath,amscd}
\usepackage[colorlinks=true]{hyperref}
\usepackage{mathrsfs}
\newtheorem{Theorem}{\indent Theorem}[section]

\newtheorem{Lemma}[Theorem]{\indent Lemma}

\theoremstyle{remark}
\newtheorem{Remark}{Remark}

\begin{document}
\centerline{   On density of the zeros of  Dedekind zeta-functions}
\medskip

\noindent
\centerline{Wei Zhang}

\centerline{
School of Mathematics and Statistics,
Henan University,
Kaifeng  475004, Henan, China}
\centerline{zhangweimath@126.com}
\medskip

{\small \textbf{Abstract}
For any $\sigma$ with $0\leq \sigma\leq 1$ and any $T>10$ sufficiently
large, let $N_{\zeta}(\sigma,K,T)$ be the number of zeros $\rho=\beta+i\gamma$ of $\zeta_{K}(s)$ with $|\gamma|\leq T$ and $\beta\geq \sigma$
and the zero being counted according to multiplicity. For $k\geq3,$ we have
\[
N_{\zeta}(\sigma,K,T)\ll T^{\frac{2k}{6\sigma-3}(1-\sigma)+\varepsilon},
\]
where
\[
\frac{2k+3}{2k+6}\leq \sigma<1
\]
and the implied constant may depend on the number field $K$ and $\varepsilon.$ This improves previous results for $k\geq3$ of certain range of $\sigma$.

\section{Introduction} Let $K$ be a number field of degree $k=[K: \mathbb{Q}]$, $k\geq2$ and $k=r_{1}+2r_{2}$, where $r_{1}$ is
the number of real conjugate fields and $2r_{2}$ is the number of complex conjugate fields.
Also let $\mathcal{O}_{K}$ be the ring of integers in $K$. For $\Re(s)>1$, one defines the Dedekind zeta
function associated to the number field $K$ as follows:
\[
\zeta_{K}(s)=\sum_{\mathfrak{a}}\frac{1}
{(\mathfrak{Na})^{s}}
=\sum_{n=1}^{\infty}\frac{a_{K}(n)}{n^{s}},\ \ \Re(s)>1,
\]
where $\mathfrak{a}$ suns over the non-zero integral ideals $\mathcal{O}_{K}$ of $K,$  $\mathfrak{Na}$ is the norm of $\mathfrak{a},$ and $a_{K}(n)=\#\{\mathfrak{a}:\mathfrak{Na}=n\}$ is the ideal counting function of $K.$
The Dedekind-zeta function has an analytic continuation to the whole complex plane, and it has only a simple pole at $s=1.$ The Dedekind-zeta function is a generalization of the Riemann zeta function $\zeta(s).$ When $K=\mathbb{Q},$ we have $\zeta(s)=\zeta_{K}(s).$
Next, we investigate estimate for the number of zeros of Dedekind zeta function $\zeta_{K}(s)$
in the strip $1/2\leq \sigma\leq 1$. For any $\sigma$ with $0\leq \sigma\leq 1$ and any $T>10$ sufficiently
large, let $N_{\zeta}(\sigma,K,T)$ be the number of zeros $\rho=\beta+i\gamma$ of $\zeta_{K}(s)$ with $|\gamma|\leq T$ and $\beta\geq \sigma$
and the zero being counted according to multiplicity. It is well known
that
\[
N_{\zeta}(0,K,T)\sim \frac{k}{\pi} T \log T
\]
as $T\rightarrow\infty$.
For $k=2$, and an absolute positive constant $C=C(\varepsilon,K)>0,$ Heath-Brown \cite{HB}  showed that
\begin{align*}
N_{\zeta}(\sigma,K,T)\ll T^{A(\sigma)(1-\sigma)}(\log T)^{C}
\end{align*}
with
\begin{align*}
A(\sigma)=
\begin{cases}
\frac{2}{\sigma}     &\textup{for}\ \  \frac{3}{4}\leq \sigma\leq 1-\varepsilon,\\
\frac{4}{3-2\sigma}       &\textup{for}\ \  \frac{1}{2}\leq\sigma\leq \frac{3}{4},
\end{cases}
\end{align*} \\
where the implied constant depends only on $K$ and $\varepsilon$.
Heath-Brown \cite{HB} also  showed that
\begin{align*}
N_{\zeta}(\sigma,K,T)\ll T^{(2+\varepsilon)(1-\sigma)}(\log T)^{C},
\end{align*}
where  $\sigma\geq {111}/{124}.$
Through out this paper, the quantity $\varepsilon$ is   an arbitrary
small quantity that is allowed to change from line to line.
In fact, for $k\geq 2,$ Heath-Brown \cite{HB} showed that $A(\sigma)=2+\varepsilon$ for $\sigma\geq3/4(1-\mu)$ provided that
$
\zeta(1/2+it)\ll t^{\mu}(\log t)^{v},
$
where $\zeta(s)$ is the Riemann zeta function. Then the Bourgain's \cite{Bo17} recent result implies that for $\sigma >63/71\approx 0.88733$, one has (by adapting the arguments of \cite{HB})
\[N_{\zeta}(\sigma,K,T)\ll T^{2(1-\sigma)
+\varepsilon}.
\]
If one adapts Ivi\'c's \cite{I2} idea for this problem, one can obtain that for
$\sigma >53/60\approx 0.88334$, one has
$N_{\zeta}(\sigma,K,T)\ll T^{2(1-\sigma)
+\varepsilon}.
$
Recently, by adapting the ideas of Bourgain \cite{Bo} and Heath-Brown \cite{H79},  Chen-Debruyne
-Vindas \cite{CDV} improves the result of Ivi\'c \cite{I2}. If one adapts the idea of \cite{CDV} for this problem, one can obtain that for
$\sigma >1407/1601\approx 0.87883$, one has
$N_{\zeta}(\sigma,K,T)\ll T^{2(1-\sigma)
+\varepsilon}.
$
The other results in Ivi\'c's \cite{I2} article cannot be translated to the case of $k=2.$ The reason is that there is no corresponding 6-th   mean value estimate for Dedekind zeta functions of $k=2.$ In fact, only by the 2-th   mean value estimate for Dedekind zeta functions of $k=2$, the sharpest   $A(\sigma)$ only be $2+\varepsilon.$ In this paper, by a different of view, by the result of \cite{H88}, we can obtain a sharper   $A(\sigma)$ when $\sigma$ is close to 1.
\begin{Theorem}\label{th1}
For $k=2,$ we have
\[
N_{\zeta}(\sigma,K,T)\ll T^{\frac{4}{6\sigma-3}(1-\sigma)+\varepsilon},
\]
where $(\kappa,\lambda)$ is any exponent  pair with
\[
\frac{1+\lambda-6\kappa}{2-8\kappa}\leq \sigma<1,
\]
$\kappa<1/4,$
$\kappa+1\leq 5\lambda/2,$ and the implied constant may depend on the number field $K$ and $\varepsilon.$
\end{Theorem}
\begin{Remark}
For $k=2$, if we choose $(\kappa,\lambda)=(1/14,11/14)$, we can obtain that for $19/20<\sigma\leq 1,$ we have
\[
N_{\zeta}(\sigma,K,T)\ll T^{\frac{4}{6\sigma-3}(1-\sigma)
+\varepsilon},
\]
where the implied constant may depend on the number field $K$ and $\varepsilon.$
\end{Remark}
For $k\geq3$ and a positive absolute constant $C$, in 1968, Sokolovsky \cite{S68} showed that
\[
N_{\zeta}(\sigma,K,T)\ll T^{(k+2-C/(k^{2}\log (k+2)))(1-\sigma)+\varepsilon}.
\]
Later, in 1977, Heath-Brown \cite{HB}  improved the above result and showed that for any  $\varepsilon>0,$ there exists a constant
$c=c(K,\varepsilon)$ such that
\[
N_{\zeta}(\sigma,K,T)\ll T^{k(1-\sigma)+\varepsilon}
\]
holds uniformly for $1/2\leq \sigma\leq 1$ when $k\geq3$.
Further assume that $\zeta_{K}(1/2+it)\ll t^{\mu}$ for some real number $\mu>0$ and for all $|t|\geq 10$.
Then, for any $1/2\leq \sigma\leq 1$, in \cite{PS}, it is proved that
\[
N_{\zeta}(\sigma,K,T)\ll T^{2(1+2\mu)(1-\sigma)}
(\log T)^{c(k)},
\]
where $c(k)$ is a positive constant depending on $k$.
As an immediate corollary, one can get
\[
N_{\zeta}(\sigma,K,T)\ll T^{((6+2k)/3)(1-\sigma)+\varepsilon},
\]
where the sub convexity results in \cite{H88} were used.
Since $(6 + 2k)/3 <k$ for any $k >6$, hence for any number field $K$ of degree $[K:Q]=k\geq7,$ the zero-density estimate  strengthens a general result of Heath-Brown \cite{HB}.
In this paper, we will show the following results. Our results improve  the result of Heath-Brown \cite{HB} of $5/6\leq \sigma \leq 1$ for $3\leq k\leq 6$ and improve  the result of  \cite{PS} for $k\geq 7.$
\begin{Theorem}\label{th3}
For $k\geq3,$ we have
\[
N_{\zeta}(\sigma,K,T)\ll T^{\frac{2k}{6\sigma-3}(1-\sigma)+\varepsilon},
\]
where
\[
\frac{2k+3}{2k+6}\leq \sigma<1
\]
and the implied constant may depend on the number field $K$ and $\varepsilon.$
\end{Theorem}

\section{Proof of Theorem \ref{th1} and \ref{th3}}

In order to estimate the number of zeros of $$\zeta_{K}(s)=\sum_{n=1}^{\infty}
\frac{a_{K}(n)}{n^{s}},\ \Re s>1,$$
where $a_{K}(n)=\sum(-1)^{\nu},$ where the sum is over all representation of $n$ as $N(\mathfrak{p}_{1}\mathfrak{p}_{2}
\cdots\mathfrak{p}_{\nu})$ with prime ideals $\mathfrak{p}_{i}$ of $K.$
By the usual procedure of zero-detecting technique will be used.  For $\zeta_{K}(s)$, we have
\[
\frac{1}{\zeta_{K}(s)}=\sum_{n=1}^\infty \frac{\mu_{K}(n)}{n^s},\ \ \Re s>1.
\]
Write
\[
M_X(s,K)=\sum_{n\leq X}\frac{\mu_{K}(n)}{n^s},
\]
where  $X>2$  is some powers of $T$ which will be chosen later and $|\mu_{K}(n)|\leq a_{K}(n)\leq n^{\varepsilon}$. By Euler product
$\mu_{K}(n)=\sum(-1)^{\nu},$ where the sum is over all representation of $n$ as $N(\mathfrak{p}_{1}\mathfrak{p}_{2}
\cdots\mathfrak{p}_{\nu})$ with distinct prime ideals $\mathfrak{p}_{i}$ of $K.$ In most of the proof of zero-density results, we choose $X=T^{\varepsilon}.$
Then we can obtain the following
\[
\zeta_{K}(s)M_X(s,K)=\sum_{n=1}^\infty \frac{C_X(n,K)}{n^s},\ \Re s>1,
\]
where
\begin{align}\label{23.100}
C_X(n,K)=
\begin{cases}
1 &\textup{if}\ n=1,\\
0 &\textup{if}\ 2\leq n\leq X,\\
D_X(n,K) & \textup{if}\ n>X
\end{cases}
\end{align}
with
\[
D_X(n,K)=\sum_{d|n,\ d\leq X}\mu_{K}(d)a_{K}\left(\frac{n}{d}\right).
\]
By similar arguments as previous section, we find that
a non-trivial zero counted of $\zeta_{K}(s)$ denoted by $N_{K}(\sigma,T)$ satisfies either
\begin{equation}\label{3.1000}
\left|\sum_{X<n<Y\log^{2}Y}
\frac{D_X(n,K)}{n^\rho}e^{-\frac nY}\right| \gg 1,
\end{equation}
or
\begin{equation}\label{23.4000}
Y^{\frac1 2-\beta}\int_{-\log^{2} T}^{\log^{2} T}\left|\zeta_{K}(\frac1 2+i\gamma+i\omega)
\right|d\omega \gg 1,
\end{equation}
$Y>2$  is some powers of $T$ which will be chosen later
Denote by  $\mathscr{R}_{1}$  the number of the class-1 zeros $\rho=\beta+i\gamma\in A$ for which (\ref{3.1000})  is satisfied, and
$\mathscr{R}_{2}$ for the class-2 zeros which satisfy (\ref{23.4000}). Moreover, we assume that any two zeros $\rho=\beta+i\gamma, \rho'=\beta'+i\gamma'$
counted in $\mathscr{R}_i$ satisfy $|\gamma-\gamma'|\geq 2\log^4T$. Then we have
\begin{align}\label{3.2000}
N_{\zeta}(\sigma,K,T) \ll (1+\mathscr{R}_{1}+\mathscr{R}_{2}) \log^{5} T.
\end{align}
\begin{Lemma}\label{le2.50000} Let $t>1.$
Then we have
$
\zeta_{K}(1/2+it)\ll t^{k/6+\varepsilon}.
$
\end{Lemma}
\begin{proof}
For $\zeta_{K}(s)$ being Dedekind zeta function associated with the algebraic number field $K,$ $[K:\mathbb{Q}]=k$, it is well known that $\zeta_{K}(1/2+it)\ll t^{k/4}.$ If K is abelian the bounds for Riemann zeta function and Dirichlet $L-$function give $\zeta_{K}(1/2+it)\ll t^{k/6+\varepsilon}$ for any $\varepsilon>0.$   In \cite{H88}, $\zeta_{K}(1/2+it)\ll t^{k/6+\varepsilon}$ was proved  for all $K,$ whether abelian or not.
\end{proof}

\begin{Lemma}\label{lea0000}
Taking $Y=T^{k/(6\sigma-3)-\varepsilon},$
then we can rule out the situation of
(\ref{23.4000}).
\end{Lemma}
\begin{proof}
For $Y=T^{k/(6\sigma-3)-\varepsilon},$ by Lemma \ref{le2.50000}
we have
\begin{align*}
&Y^{\frac1 2-\beta}\int_{-\log^{2} T}^{\log^{2} T}\left|\zeta_{K}(\frac1 2+i\gamma+i\omega)
\right|d\omega
\\&\ll T^{-\varepsilon}.
\end{align*}
This completes the proof.
\end{proof}

The following lemma  can be proven by section 3.
\begin{Lemma}\label{leb000}   For $1/2<\sigma<1,$ we have
\begin{equation}
\mathscr{R}_0 \ll T^{\varepsilon}\left(N^{2-2\sigma}+
\mathscr{R}_0N^{1+\lambda-\kappa-2\sigma}
T^{\kappa+\varepsilon}\right),
\end{equation}
for some $N$ satisfying
\begin{equation}\label{23.700}
 Y^{\frac12} \log Y<N\leq Y\log^2Y,
\end{equation}
where $(\kappa,\lambda)$ is any exponent pair.
\end{Lemma}

For $k=2,$ to bound $\mathscr{R}_1$, the interval $[-T, T]$ can be  divided into consecutive  intervals of length $T_0$. Denote by  $\mathscr{R}_0$ the number of class-1 zero in $\mathscr{R}_1$  with $|t|\leq T_0$.  Then  Huxley's subdivision \cite{Hu}, we have
$$
\mathscr{R}_{1}\ll \mathscr{R}_{0}\left(1+\frac{T}{T_{0}}\right).
$$

Set $$T_{0}=N^{\frac{2\sigma-1-(\lambda-\kappa)}
{\kappa}+\varepsilon}.$$
By  Lemma \ref{leb000}, one can obtain
\[
\mathscr{R}_{0} \ll T^{\varepsilon}N^{2-2\sigma}.
\]
Therefore,
we have
\begin{align*}
\mathscr{R}_{1}\ll  T^{\varepsilon}N^{2-2\sigma}\left(1+\frac{T}{T_{0}}\right)\ll T^{\varepsilon}\left(N^{2-2\sigma}+TN^{2-2\sigma-\frac{2\sigma-1-(\lambda-\kappa)}{\kappa}}\right).
\end{align*}
Setting
$$Y=T^{2/{(6\sigma-3)}-\varepsilon},$$
then for $k=2,$ $$\frac{1+\lambda+\kappa}{2(1+\kappa)}\leq \sigma\leq1,$$
  and
$$\frac{1+\lambda-6\kappa}{2-8\kappa}\leq \sigma\leq 1,$$
 we get
\begin{align*}
\begin{split}
\mathscr{R}_1+\mathscr{R}_2&\ll
T^{\varepsilon}\left(N^{2-2\sigma}+TN^{2-2\sigma-\frac{2\sigma-1-(\lambda-\kappa)}{\kappa}}
+1\right)\\
&\ll
T^{\varepsilon}\left(Y^{2-2\sigma}+T
Y^{1-\sigma-\frac{2\sigma-1-(\lambda-\kappa)}{2\kappa}}
+1\right)\\
&\ll
T^{\varepsilon}
\left(T^{\frac{4(1-\sigma)}{6\sigma-3}}
+T^{1+\frac{2\left(1- \sigma-\frac{2\sigma-1-
(\lambda-\kappa)}{2\kappa}\right)}
{6\sigma-3}}\right)\\
&\ll
T^{\frac{4(1-\sigma)}{6\sigma-3}+\varepsilon},
\end{split}
\end{align*}
where $\kappa<1/4.$
Note that when $\kappa+1\leq 5\lambda/2,$ we have
\[
\frac{1+\lambda-6\kappa}{2-8\kappa}\geq \frac{1+\lambda+\kappa}{2(1+\kappa)}.
\]
This completes the proof of Theorem \ref{th1}.

For $k\geq3,$ to bound $\mathscr{R}_1$, the interval $[-T, T]$ can be  divided into consecutive  intervals of length $T_0$. Denote by  $\mathscr{R}_0$ the number of class-1 zero in $\mathscr{R}_1$  with $|t|\leq T_0$.  Then  Huxley's subdivision \cite{Hu}, we have
$$
\mathscr{R}_{1}\ll \mathscr{R}_{0}\left(1+\frac{T}{T_{0}}\right).
$$
Huxley's subdivision method is hte technique of dividing $T$ into subintervals of length $T_{0}$ and then multiplying the obtained estimate by $(1+T/T_{0}).$
Set $$T_{0}=N^{ 4\sigma-2 +\varepsilon}.$$
By  Lemma \ref{leb000}, choosing $(\kappa,\lambda)=(1/2,1/2),$ one can obtain
\[
\mathscr{R}_{0} \ll T^{\varepsilon}N^{2-2\sigma}.
\]
Therefore, for
$$\frac23\leq \sigma\leq1,$$
we have
\begin{align*}
\mathscr{R}_{1}\ll  T^{\varepsilon}N^{2-2\sigma}
\left(1+\frac{T}{T_{0}}\right)\ll T^{\varepsilon}\left(N^{2-2\sigma}+TN^{4-6\sigma}\right).
\end{align*}
Recall that
$$Y=T^{k/{(6\sigma-3)}-\varepsilon},$$
by Lemma \ref{lea0000}, we can obtain
 \[
 \mathscr{R}_{1}+ \mathscr{R}_{2}\ll  T^{\frac{2k(1-\sigma)}{6\sigma-3}
 +\varepsilon},
 \]
 provided that $k\geq3$ for
\[
\sigma \leq 1.
\]
As (otherwise the estimate is trivial)
 \[
\frac{2k(1-\sigma)}{6\sigma-3}\leq 1\]
for $\frac{2k+3}{2k+6}\leq \sigma \leq 1,$
we have
\[
 \mathscr{R}_{1}+ \mathscr{R}_{2}\ll  T^{\frac{2k(1-\sigma)}{6\sigma-3}
 +\varepsilon},
 \]
provided that
\[
 \max\left\{\frac23,
\frac{2k+3}{2k+6}\right\}\leq \sigma \leq 1.
 \]
This completes the proof of Theorem \ref{th3}.
 $\square$

\section{Proof of Lemma \ref{leb000}}

To prove Lemma \ref{leb000},  we will follow the argument in  section 11 of \cite{I} and \cite{R19}.  By dyadic interval we  split the sum in into $O(\log Y)$ sums of the form
$\sum_{M<n\leq 2M}\frac{D_X(n)}{n^\rho}e^{-\frac nY} .
$
Hence each $\rho=\beta+i\gamma$ counted in $R_{1}$ satisfies
\begin{align}\label{M}
\sum_{M_j<n\leq 2M_j}\frac{D_X(n)}{n^\rho}e^{-\frac nY} \gg \frac{1}{\log Y}
\end{align}
for some $M_j$ satisfying
$$
X \leq M_j =2^{-j}Y\log^{2} Y\leq Y\log^{2} Y, \ \ j =1, 2,..., J,
$$
where $J= [\log(X^{-1}Y\log^2Y)/(\log 2)]+1$.  Denote by $\mathcal R_j$ the number of zeros satisfying (\ref{M}) in $R_1$, and let $\mathcal{R}=\mathcal R_{j_0}$ be the largest among $\mathcal R_j$ for $1\leq j\leq J$. Then
\begin{align}\label{24.2}
\mathscr{R}_1\ll \mathcal R \log Y.
\end{align}
Write $M=M_{j_0}$ and choose  the large enough positive integer $l\geq l_{1}\geq1$ such that ($l$ is a natural number depended on $M$ and $l_{1}$ is a fixed integer)
\begin{align*}
M^{l} \ll Y^{l_{1}} \log^{l_{1}} Y \ll M^{l+1},
\end{align*}
and then raise both sides of (\ref{M}) to the power $l$, we get
\begin{align*}
\left|\sum_{M^{l}<n\leq (2M)^{l}}\frac{E_X(n)}{n^{\rho}}\right| \gg \frac{1}{\log^{l} Y},
\end{align*}
where
\[
E_X(n)=\sum_{n_{1}n_{2}\cdots n_{l}=n\atop M<n_{i}\leq 2M}D_X(n_{1})D_X(n_{2})\cdots D_X(n_{l})e^{-\frac {n_{1}+n_{2}+\cdots+n_{l}}{Y}}.
\]
Write $N=M^{l}, N_1=(2M)^l$. Then
\begin{align*}
Y^{\frac{l_{1}^{2}}{l_{1}+1}} \log^{\frac{2l_{1}^{2}}{l_{1}+1}} Y\ll N \ll Y^{l_{1}} \log^{2l_{1}} Y.
\end{align*}
As the upper bound of  2-th power moment for the cusp forms, we know that  choosing $l_{1}=1$ is suitable.
Then the number
 of $\rho$ which  satisfies
\begin{align}\label{2E}
\left|\sum_{N<n\leq N_1}\frac{E_X(n)}{n^{\rho}}\right| \gg \frac{1}{\log^{l} Y}
\end{align}
can be   counted by $\mathcal R.$
Now we show that
\begin{align}\label{24.4}
\sum_{N<n\leq N_1}\frac{\left| E_X(n)\right|^{2}}{n^{2\sigma}}
\ll N^{1-2\sigma+\varepsilon}.
\end{align}
Then we have
\begin{align*}
\begin{split}
\sum_{N<n\leq N_1}\frac{\left| E_X(n)\right|^{2}}{n^{2\sigma}}
\ll N^{1-2\sigma+\varepsilon}.
\end{split}
\end{align*}
To bound $\mathcal R$, we let $\rho=\beta_r+i\gamma_r$ in (\ref{2E}), where  $\sigma\leq \beta_r\leq 1,\ |\gamma_r|\leq T$. Then
\[\mathcal R\ll \log^{l}Y \sum_{r\leq \mathcal R}\left|\sum_{N<n\leq N_1}\frac{E_X(n)}
{n^{\beta_r+i\gamma_r}}\right| = \log^{l}Y\sum_{r\leq \mathcal R}\left|\sum_{N<n\leq N_1}\frac{E_X(n)}
{n^{\sigma+i\gamma_{r}}}
n^{\sigma-\beta_r}\right|.
\]
By partial summation we get
\begin{align} \label{2G}
\mathcal R \ll  \log^{l}Y \max_{N<u \leq N_1}\sum_{r\leq \mathcal R}\left|
\sum_{N<n\leq u}\frac{E_X(n)}{n^{\sigma+i\gamma_{r}}}
\right|.
\end{align}
To bound the above double sum, we will use the following well known lemma, which can also be seen in \cite{I,T,Mo}.
 \begin{Lemma}\label{le2.2} {\rm (Hal\'{a}sz-Montgomery inequality)} Let  $a_n,b_n \ (n=1,\ 2, ... )$ be complex numbers. For $a=\{a_{n}\}_{n=1}^{\infty}\in\mathbb{C}^\infty$, $\ b=\{b_{n}\}_{n=1}^{\infty}\in \mathbb{C}^\infty$, define the inner product
\[
(a,b)=\sum_{n=1}^{\infty}a_{n}\overline{b}_{n},\quad \|a\|^{2}=(a,a).
\]
Let $\xi,\varphi_{1},\varphi_{2},...,\varphi_{R}$ be arbitrary vectors in $\mathbb{C}^\infty$. Then we have
\begin{align}
\sum_{r\leq R}\left|\left(\xi, \varphi_{r}  \right)  \right|
\leq\|\xi\| \bigg(\sum_{r,s \leq R}(\varphi_{r}, \varphi_{s} )\bigg)^{\frac{1}{2}}.
\end{align}
 \end{Lemma}
For $\xi= \{\xi_{n}\}_{n=1}^{\infty}$, by Lemma \ref{le2.2}, we have
\begin{align*}
\xi_{n}=
\begin{cases}
E_{X}(n)e^{\frac{n}{2N}}n^{-\sigma} &\textup{if}\ N<n \leq u,\\
0   & otherwise,
\end{cases}
\end{align*}
and $\varphi_{r}= \{\varphi_{r,n}\}_{n=1}^{\infty}$ with
$\varphi_{r,n}=e^{-\frac{n}{2N}}n^{-i\gamma_{r}}$.
By  (\ref{24.4}) one has $||\xi||^{2} \ll N^{1-2\sigma+\varepsilon}$.
Thus by (\ref{2G}) and Lemma \ref{le2.2} we get
\begin{align}\label{I}
\mathcal R ^2\ll Y^{\varepsilon}\left(\mathcal R N^{2-2\sigma}
+N^{1-2\sigma}\sum_{r\neq s \leq \mathcal R }\left|G(i\gamma_{r}-i\gamma_{s})   \right|\right).
\end{align}
where
$
G(it)=\sum_{n=1}^{\infty}e^{-\frac{n}{N}}n^{-it}.
$
Here we have used the fact  that $G(0)\ll N$ for bounding the diagonal terms.
By
Mellin's transform, we have
\[
\sum_{n=1}^{\infty}e^{-\frac{n}{Y}}n^{-s}=
(2\pi i)^{-1}\int_{2-i\infty}^{2+i\infty}\zeta(s+\omega)\Gamma(\omega)
Y^{\omega}d\omega.
\]
Thus
\begin{align*}
G(it)=
(2\pi i)^{-1}\int_{2-i\infty}^{2+i\infty}\zeta(\omega+it)\Gamma(\omega)
N^{\omega}d\omega.
\end{align*}
Let $(\kappa,\lambda)$ be any exponent pair.
Then move the  integration line  to $\Re(\omega)=\lambda-\kappa$ for $(\kappa,\lambda)\neq (1/2,1/2)$ and to $\Re(\omega)=\lambda-\kappa+\varepsilon$ for $(\kappa,\lambda)=(1/2,1/2)$, and note  that the residue at the simple pole $\omega=1-it$ is $O\left(Ne^{-|t|}\right)$, we get
\begin{align*}
G(it)=
(2\pi i)^{-1}\int_{\lambda-\kappa-i\infty}^{\lambda-\kappa+i\infty}\zeta(\omega+it)\Gamma(\omega)
N^{\omega}d\omega+O\left(Ne^{-|t|}\right).
\end{align*}
By Stiring's formula, it is easy to check that  the contribution from  $|v| \geq \log ^{2} T$ to the above integral is $O(1)$, where $v=\Im\omega$. Moreover for  $|v| \leq \log ^{2} T$, one has $\Gamma(\lambda-\kappa+iv)\ll 1$. Therefore, we get
\begin{align*}
\begin{split}
\sum_{r\neq s \leq \mathcal R}&\left|G\left(i\gamma_{r}-i\gamma_{s}\right)
\right|\\
\ll & N\sum_{r\neq s \leq \mathcal R}e^{-\left|\gamma_{r}-\gamma_{s}\right|}
+\mathcal R^{2}+N^{\lambda-\kappa}\int_{-\log^{2} T}^{\log^{2} T} \sum_{r\neq s \leq \mathcal R}\left|\zeta\left(\lambda-\kappa+ i\gamma_{r}-i\gamma_{s}+iv\right)\right|dv\\
\ll & \mathcal R N+\mathcal R^{2}+N^{\lambda-\kappa}\mathcal R^{2}T^{\kappa+\varepsilon} ,
\end{split}
\end{align*}
where we have used
\[
\zeta(\lambda-\kappa+it)\ll (|t|+1)^{\kappa+\varepsilon},
\]
which can be seen in \cite{I,T}.
Back to (\ref{I}), we get
\begin{align*}
\begin{split}
\mathcal R&\ll T^{\varepsilon}\left(N^{2-2\sigma}+\mathcal R+
N^{1-2\sigma}N^{\lambda-\kappa}\mathcal RT^{\kappa+\varepsilon}
\right)\\
&\ll T^{\varepsilon}\left(N^{2-2\sigma}+ \mathcal R+
\mathcal R N^{1+\lambda-\kappa-2\sigma}T^{\kappa+\varepsilon}\right).
\end{split}
\end{align*}

\subsection*{Acknowledgements} I would like to thank the referee for going through the manuscript and valuable suggestions.




\begin{thebibliography}{99}






\bibitem{Bo}
 J. Bourgain, {\it On large values estimates for Dirichlet polynomials
and the density hypothesis for the Riemann zeta function}.
IMRN  3(2000), 133-146.

\bibitem{Bo17}
 J. Bourgain, {\it Decoupling, exponential sums and the Riemann zeta function.} J. Amer. Math. Soc. 30 (2017), 205-224.

\bibitem{CDV}
B. Chen,
G. Debruyne,
J. Vindas,
{\it On the density hypothesis for L-functions associated with holomorphic cusp forms.} Rev. Mat. Iberoam. (2024), published online first






\bibitem{HB}
D.R. Heath-Brown, {\it On the density of the zeros of the Dedekind zeta-function},  Acta Arith., (1977), 33: 169-181.






\bibitem{H88}
D.R. Heath-Brown {\it The growth rate of the Dedekind zeta-function on the critical line.}
Acta Arith. 49, No. 4, 323-339 (1988)

\bibitem{H79}
D.R. Heath-Brown {\it
A large values estimate for Dirichlet polynomials.}
J. London Math. Soc. (2) 20 (1979), no. 1, 8-18.
\bibitem{Hu}
M. N. Huxley, {\it On the difference between consective primes}.
Invent. Math.  15(1972), 164-170.




\bibitem{I}
A. Ivi\'{c}, {\it The Riemann Zeta-function. The Theory of Riemann Zeta-function
with Applications}. John Wiley, New York, 1985.


\bibitem{I2}
A. Ivi\'{c}, {\it On zeta-functions associated with Fourier coefficients of cusp forms}, in: Proceedings of the Amalfi Conference on Analytic Number Theory, E. Bombieri et
al. (eds.), Universit\'{a} di Salerno, 1992, pp.231-246.







\bibitem{Mo}
H. L. Montgomery, {\it Topics in Multiplicative Number Theory}, Springer, Berlin, 1971.


\bibitem{PS}
B. Paul and A. Sankaranarayanan, {\it On the error term and zeros of the Dedekind zeta function.} J. Number Theory 215 (2020), 98-119.




\bibitem{R19}
X.M. Ren, W. Zhang, {\it
Zero-density estimate of L-functions for cusp forms.}
J. Number Theory 194(2019), 284-296.


\bibitem{S68}
A.V. Sokolovsky, {\it
A theorem on the zeros of Dedekind’s zeta-function and the distance between ``neighboring'' prime ideals.}
Acta Arith. 13, 321-334 (1968).
\bibitem{T}
E.C. Titchmarsh, {\it The Theory of the Riemann Zeta-Function}, 2nd. edn  University Press, Oxford 1986.




\end{thebibliography}
\end{document}